\documentclass{amsart}
\usepackage{amsmath,amssymb,amsthm,url}
\newtheorem{thr}{Theorem}
\newtheorem{lem}[thr]{Lemma}
\newtheorem{conj}[thr]{Conjecture}
\newtheorem{cor}[thr]{Corollary}
\newtheorem{prop}[thr]{Proposition}
\newtheorem{obs}[thr]{Observation}
\newtheorem{claim}[thr]{Claim}

\theoremstyle{definition}
\newtheorem{defn}[thr]{Definition}

\theoremstyle{remark}

\numberwithin{equation}{section}

\def\C{\mathbb{C}}

\def\M{\mathcal{M}}

\def\F{\mathcal{F}}
\def\Fb{\mathbb{F}}
\def\rank{\operatorname{rank}}

\begin{document}

\title[A counterexample to Strassen's direct sum conjecture]{A counterexample to Strassen's \\ direct sum conjecture}

\author{Yaroslav Shitov}
\email{yaroslav-shitov@yandex.ru}







\maketitle

The multiplicative complexity of systems of bilinear forms (and, in particular, the famous question of fast matrix multiplication) is an important area of research in modern theory of computation. One of the foundational papers on the topic is Strassen's work~\cite{Stra69}, which contains an $O(n^{\ln7/\ln2})$ algorithm for the multiplication of two $n\times n$ matrices. 
In his subsequent paper~\cite{Strassen} published in 1973, Strassen asked whether the multiplicative complexity of the union of two bilinear systems depending on different variables is equal to the sum of the multiplicative complexities of both systems. A stronger version of this problem was proposed in the 1981 paper~\cite{FW} by Feig and Winograd, who asked whether any optimal algorithm that computes such a pair of bilinear systems must compute each system separately. These questions became known as the \textit{direct sum conjecture} and \textit{strong direct sum conjecture}, respectively, and they were attracting a notable amount of attention during the four decades. As Feig and Winograd wrote, `\textit{either a proof of, or a counterexample to, the direct sum conjecture will be a major step forward in our understanding of complexity of systems of bilinear forms}.'

The modern formulation of this conjecture is based on a natural representation of a bilinear system as a three-dimensional \textit{tensor}, that is, an array of elements $T(i|j|k)$ taken from a field $\F$, where the triples $(i,j,k)$ run over the Cartesian product of finite \textit{indexing sets} $I, J, K$. A tensor $T$ is called \textit{decomposable} if $T=a\otimes b\otimes c$ (which should be read as $T(i|j|k)=a_ib_jc_k$), for some vectors $a\in\F^I$, $b\in\F^J$, $c\in\F^K$. The \textit{rank} of a tensor $T$, or the \textit{multiplicative complexity} of the corresponding bilinear system, is the smallest $r$ for which $T$ can be written as a sum of $r$ decomposable tensors with entries in $\F$. We denote this quantity by $\rank_\F T$, and we note that the rank of a tensor may change if one allows to take the entries of decomposable tensors as above from an extension of $\F$, see~\cite{Berg}. Taking the union of two bilinear systems depending on disjoint sets of variables corresponds to the \textit{direct sum} operation on tensors. More precisely, if $T$ and $T'$ are tensors with disjoint indexing sets $I,I',J,J',K,K'$, then we can define the direct sum $T\oplus T'$ as a tensor with indexing sets $I\cup I'$, $J\cup J'$, $K\cup K'$ such that the $(I|J|K)$ block equals $T$ and $(I'|J'|K')$ block equals $T'$, and all entries outside of these blocks are zero. In other words, direct sums of tensors are a multidimensional analogue of block-diagonal matrices; a basic result of linear algebra says that the ranks of such matrices are equal to the sums of the ranks of their diagonal blocks. Strassen's direct sum conjecture is a three-dimensional analogue of this statement.

\begin{conj}\label{ConStrassen}\textup{(See~\cite{Strassen}.)}
If $T_1, T_2$ are tensors over an infinite field $\F$, then $$\rank_\F(T_1\oplus T_2)=\rank_\F\, T_1+\rank_\F\, T_2.$$
\end{conj}

The `$\leqslant$' inequality in Conjecture~\ref{ConStrassen} is obvious because a direct sum of decompositions of $T_1$ and $T_2$ is a decomposition of $T_1\oplus T_2$. If every optimal decomposition of $T_1\oplus T_2$ arises in this way, then $(T_1,T_2)$ are said to satisfy the \textit{strong direct sum conjecture}, see~\cite{AFW, Bshouty2}. Both Conjecture~\ref{ConStrassen} and its strong form remained open to this date, and the early progress on these problems included the papers~\cite{Bshouty, dGH} solving them in different special cases using a specific correspondence of bilinear systems and associative algebras over $\F$. In 1986, JaJa and Takshe (see~\cite{JJT}) proved Conjecture~\ref{ConStrassen} in the case when one of the dimensions of either $T_1$ or $T_2$ is at most two, and this result was subsequently generalized and discussed from the point of view of algebraic geometry by Landsberg and Micha\l{}ek in~\cite{LMat}. The geometric approach lead to many new sufficient conditions for the direct sum conjecture (see~\cite{Teit}) and presented several interesting equivalent formulations, generalizations, and relaxations of these conjectures in the language of algebraic geometry, see~\cite{BGL,CCO, LandBook}.

\section{Preliminaries}

In our paper, we construct a counterexample to the direct sum conjecture. We begin with several additional definitions, known results, and notational conventions.

\subsection{Notation.} Let $\F$ be a field; we denote by $\overline{\F}$ the algebraic closure of $\F$. For a tensor $T$ in $\F^{I\times J\times K}$ and $k\in K$, we define the $k$th \textit{$3$-slice} as a matrix in $\F^{I\times J}$ whose $(i,j)$ entry equals $T(i|j|k)$. For all $i\in I$, $j\in J$, we can define the $i$th $1$-slice and $j$th $2$-slice of $T$ in a similar way. 
The \textit{support} of $T$ is the smallest set $I_0\times J_0\times K_0\subset I\times J\times K$ containing all the non-zero entries of $T$, and the sets $I_0$, $J_0$, $K_0$ are called the $1$-, $2$-, and $3$-supports of $T$. Two tensors are called \textit{equivalent} if they become equal when restricted to their supports. If $\Sigma$ is a non-empty finite set, then we define the $\Sigma$-\textit{clone} of $T$ as the tensor $T_\Sigma$ obtained from $T$ by taking $|\Sigma|$ copies of every element in every indexing set. Namely, we define the tensor $T_\Sigma$ with indexing sets $\Sigma\times I$, $\Sigma\times J$, $\Sigma\times K$ as $T_\Sigma(s_1,i|s_2,j|s_3,k)=T(i|j|k)$ for all $s_1,s_2,s_3\in\Sigma$. Clearly, taking the clone of a tensor does not change its rank. All these notions can be defined for matrices instead of tensors in an analogous way.



\subsection{Eliminating rank-one slices.} 
Now let $V_1\subset\F^{J\times K}$, $V_2\subset\F^{I\times K}$, $V_3\subset\F^{I\times J}$ be $\F$-linear spaces consisting of matrices. By $T\,\mathrm{mod}\,(V_1,V_2,V_3)$ we denote the set of all tensors that can be obtained from $T$ by adding elements of $V_1$ to the $1$-slices of $T$, elements of $V_2$ to the $2$-slices of the resulting tensor, and elements of $V_3$ to the $3$-slices of what we obtained after adding the $1$- and $2$-slices. The following statement is well known, see e.g. Lemma~2 in~\cite{HK} and Proposition~3.1 in~\cite{LMat}.

\begin{lem}\label{lemcompl}
Let $T\in\F^{I\times J\times K}$ and $K=\{1,\ldots,k\}\cup\{1',\ldots,k'\}$. Denote an $i$th $3$-slice of $T$ by $S_i$ and let $V$ be the $\F$-linear span of $S_{1'},\ldots,S_{k'}$. Then
$$\operatorname{rank}_\F T\geqslant \min \operatorname{rank}_\F T\,\mathrm{mod}(0,0,V)+\dim V,$$ and if $S_{1'},\ldots,S_{k'}$ are also rank-one, then the equality holds.
\end{lem}

Now let finite sets $\M_1,\M_2,\M_3$ 
be bases of the linear spaces $V_1,V_2,V_3$ as above. We define the tensor $\mathcal{T}$ with indexing sets $I\cup\M_1$, $J\cup\M_2$, $K\cup\M_3$ as 

\noindent (1) $\mathcal{T}(\alpha|\beta|\gamma)=T(\alpha|\beta|\gamma)$ if $(\alpha,\beta,\gamma)\in I\times J\times K$;

\noindent (2) for any $\chi\in\{1,2,3\}$ and any $m\in\M_\chi$, the $m$th $\chi$-slice of $\mathcal{T}$ is equivalent to $m$ (that is, coincides with $m$ up to adding zero rows and columns).

We say that $\mathcal{T}$ is obtained from $T$ by \textit{adjoining} the $1$-slices $\M_1$, the $2$-slices $\M_2$, and the $3$-slices $\M_3$, or simply by adjoining $(\M_1,\M_2,\M_3)$. The result below follows from Lemma~\ref{lemcompl}.

\begin{lem}\label{lemcompl5}
Let $T$, $\mathcal{T}$, $V_1$, $V_2$, $V_3$ be as above in this subsection. Then 
$$\operatorname{rank}_\F\mathcal{T}\geqslant \min \operatorname{rank}_\F T\,\mathrm{mod} (V_1,V_2,V_3)+\dim V_1+\dim V_2+\dim V_3,$$ and if the matrices in the $\M_i$'s are rank-one, then the equality holds.
\end{lem}

\subsection{The main result and our strategy.}

We are going to refute Conjecture~\ref{ConStrassen} by proving the following result.

\begin{thr}\label{thrmain2}
For any infinite field $\F$, there exist tensors $T_1$, $T_2$ over $\F$ such that $$\rank_\F(T_1\oplus T_2)<\rank_{\overline{\F}}\, T_1+\rank_{\overline{\F}}\, T_2.$$
\end{thr}

Of course, the assertion with $\overline{\F}$ replaced by $\F$ in the summands of the right-hand side can only be weaker than the initial one, so Theorem~\ref{thrmain} disproves Strassen's conjecture in the original formulation, that is, over any infinite field. Our counterexample rests on the following two claims, which we prove in subsequent sections.

\begin{claim}\label{claimmain1}
Let $T\in \F^{I\times J\times K}$ and assume 
$W_1,W_2,W_3$ are $\overline{\F}$-linear subspaces of $J\times K$, $I\times K$, $I\times J$ matrices, respectively. We assume that $W_1,W_2,W_3$ have bases consisting of matrices with entries in $\F$. Then there are a finite set $\Sigma$ and $$\M_1\subset \F^{(\Sigma\times J)\times(\Sigma\times K)},\,\,\,\,\M_2\subset \F^{(\Sigma\times I)\times(\Sigma\times K)},\,\,\,\,\M_3\subset \F^{(\Sigma\times I)\times(\Sigma\times J)}$$ consisting of finitely many rank-one matrices such that 

\noindent (1) $\operatorname{span}\M_\delta$ contains the $\Sigma$-clone of every $w_\delta\in W_\delta$, and

\noindent (2) the tensor $\mathcal{T}$ obtained by adjoining $(\M_1,\M_2,\M_3)$ to the $\Sigma$-clone of $T$ satisfies $\rank_{\overline{\F}}\,\mathcal{T}=\min\rank_{\overline{\F}}\,T\,\mathrm{mod}(W_1,W_2,W_3)+\sum_{\delta=1}^3\dim\operatorname{span}\M_\delta.$
\end{claim}

\begin{claim}\label{claimmain2}
For any infinite field $\F$, there is $T\in\F^{(I_1\cup I_2)\times (J_1\cup J_2)\times (K_1\cup K_2)}$ such that
$$\rank_\F T<\min\rank_{\overline{\F}}\,T_{111}\,\mathrm{mod}(U_1,U_2,U_3)+\min\rank_{\overline{\F}}\,T_{222}\,\mathrm{mod}(V_1,V_2,V_3),$$
where the indexing sets $I_1,I_2,J_1,J_2,K_1,K_2$ are disjoint, the tensor $T_{ijk}$ is the $(I_i|J_j|K_k)$ block of $T$, and 
$U_1$, $U_2$, $U_3$, $V_1$, $V_2$, $V_3$ are the $\overline{\F}$-linear spaces spanned, respectively, by the $1$-slices of $T_{211}$, by the $2$-slices of $T_{121}$, by the $3$-slices of $T_{112}$, by the $1$-slices of $T_{122}$, by the $2$-slices of $T_{212}$, and by the $3$-slices of $T_{221}$.
\end{claim}

\begin{thr}\label{thrmain}
Claims~\ref{claimmain1} and~\ref{claimmain2} imply Theorem~\ref{thrmain2}.
\end{thr}

\begin{proof}
We set $\mathcal{T}_1$ and $M_\delta^1$ to be the tensor and matrix sets obtained by the application of Claim~\ref{claimmain1} to the tensor $T_{111}$ and linear spaces $U_1,U_2,U_3$ as in Claim~\ref{claimmain2}. We define $\mathcal{T}_2$ and $M_\delta^2$ similarly but with $T_{222},V_1,V_2,V_3$ taken instead of $T_{111},U_1,U_2,U_3$, and we assume without loss of generality that the sets $\Sigma$ arisen from the applications of Claim~\ref{claimmain1} to $T_{111}$ and $T_{222}$ are equal. We write $\rho=\rank_{\F}(\mathcal{T}_1\oplus\mathcal{T}_2)$, and we denote by $\mathcal{W}_\delta$ the ${\F}$-linear space spanned by the set of all matrices equivalent to those in $M_\delta^1$ and $M_\delta^2$ of sizes corresponding to the $\delta$-slices of $\mathcal{T}_1\oplus\mathcal{T}_2$.

We prove Theorem~\ref{thrmain2} by checking that $\rho<\rank_{\overline{\F}} \mathcal{T}_1+\rank_{\overline{\F}} \mathcal{T}_2$, or
\begin{equation}\label{eq10}\rho-D<\min\rank_{\overline{\F}} T_{111}\,\mathrm{mod}\,(U_1,U_2,U_3)+\min\rank_{\overline{\F}} T_{222}\,\mathrm{mod}\,(V_1,V_2,V_3),\end{equation}
where $D=\sum_{\delta=1}^3\dim\mathcal{W}_\delta$. Since $\mathcal{T}_1\oplus\mathcal{T}_2$ is obtained from $T_{\Sigma1}\oplus T_{\Sigma2}$  by adjoining the bases of linear spaces $(\mathcal{W}_1,\mathcal{W}_2,\mathcal{W}_3)$, the left-hand side of~\eqref{eq10} equals \begin{equation}\label{eq11}\min\rank_\F(T_{\Sigma1}\oplus T_{\Sigma2})\,\mathrm{mod}\,(\mathcal{W}_1,\mathcal{W}_2,\mathcal{W}_3)\end{equation}
by Lemma~\ref{lemcompl5} (where $T_{\Sigma\chi}$ denotes the $\Sigma$-clone of $T_{\chi\chi\chi}$).

According to the item (1) of Claim~\ref{claimmain1}, the matrices equivalent to the $\Sigma$-clones of the $1$-slices of $T_{122}$ and $T_{211}$ belong to $\mathcal{W}_1$, the matrices equivalent to the $\Sigma$-clones of the $2$-slices of $T_{121}$ and $T_{212}$ belong to $\mathcal{W}_2$, and the matrices equivalent to the $\Sigma$-clones of the $3$-slices of $T_{112}$ and $T_{221}$ belong to $\mathcal{W}_3$, so we see that the $\Sigma$-clone of $T$ belongs to $(T_{\Sigma1}\oplus T_{\Sigma2})\,\mathrm{mod}\,(\mathcal{W}_1,\mathcal{W}_2,\mathcal{W}_3)$. In other words, the value~\eqref{eq11} does not exceed $\rank_\F\, T$, and since this value equals $\rho-D$, we complete the proof of~\eqref{eq10} by applying the conclusion of Claim~\ref{claimmain2}.
\end{proof}


Now we need to check the validity of Claims~\ref{claimmain1} and~\ref{claimmain2}; in the following section, we prove the former of them by developing the construction as in~\cite{myComon}. In Section~3, we prove Claim~\ref{claimmain2} by using results on ranks of generic tensors and sufficient conditions of algebraic independence in fields.

\section{The proof of Claim~\ref{claimmain1}}

In this section, we follow the notation of Claim~\ref{claimmain1} and work over a field $\F$.
Let $T\in\F^{I\times J\times K}$ be a tensor as in Claim~\ref{claimmain1}; we begin by constructing the corresponding tensor $\mathcal{T}$ in the special case when $W_1=W_2=0$ and $W_3$ is the $1$-dimensional subspace spanned by a rank-$r$ matrix $W$ which has ones at the positions $(a_1,b_1),\ldots,(a_r,b_r)\in I\times J$ and zeros everywhere else. We set $\rho=2|I\times J\times K|+1$, $\theta=\lceil\log_2r\rceil$, and $\sigma=2\rho^2 r$. We define the trivial partition of a $\sigma\times\sigma$ matrix into $\sigma^2$ submatrices of size $1\times 1$ as follows. Any number $s\in\{0,\ldots,\sigma^2-1\}$ can be written as $s=u_{1s}\sigma+u_{2s}$ with $u_{1s},u_{2s}\in\{0,\ldots,\sigma-1\}$. We will write $u(1,s)=\{u_{1s}\}$, $u(2,s)=\{u_{2s}\}$. Clearly, the $\sigma^2$ sets $u(1,t)\times u(2,t)$ are a trivial partition of a $\sigma\times\sigma$ matrix. We proceed with the definition of the sets $\Sigma$ and $\M_\delta$ as in Claim~\ref{claimmain1}.


\begin{defn}\label{defmd}
We set $\Sigma=\{0,\ldots,\sigma-1\}^\theta$; we set $\pi^i$ to be a function from $\{1,\ldots,\theta\}$ to $\{1,2\}$ and assume that these functions are pairwise different for $i\in\{1,\ldots, r\}$. For any function $\Phi$ from $\{1,\ldots,\theta\}$ to $\{0,\ldots,\sigma^2-1\}$, we define the $0$-$1$ matrix $M_\Phi\in\F^{(I\times\Sigma)\times (J\times\Sigma)}$ with $1$-support equal to
$$\bigcup\limits_{i=1}^r \{a_i\}\times u(\pi^i_1,\Phi_1)\times\ldots\times u(\pi^i_\theta,\Phi_\theta),$$
and $2$-support equal to
$$\bigcup\limits_{j=1}^r \{b_j\}\times u(3-\pi^j_1,\Phi_1)\times\ldots\times u(3-\pi^j_\theta,\Phi_\theta).$$
We call an $(a_i\times\Sigma)\times(b_j\times\Sigma)$ block of $M_\Phi$ diagonal if $i=j$.
We define $\M$ as the set of all such $M_\Phi$ and those matrix units that correspond to entries which are non-zero in at least one of the non-diagonal blocks in any of the $M_\Phi$. 
\end{defn}

\begin{obs}\label{obssum}
The diagonal blocks of $\sum_\Phi M_\Phi$ are matrices of all ones.
\end{obs}

\begin{proof}
If $\{S_1^1,\ldots,S_k^1\}$, $\ldots\,$, $\{S_1^q,\ldots,S_k^q\}$ are partitions of a set $S$, then $\{S_1^1,\ldots,S_k^1\}\times\ldots\times\{S_1^q,\ldots,S_k^q\}$ is a partition of $S^q$.
\end{proof}

\begin{obs}\label{obsspan}
The $\Sigma$-clone of $W$ lies in the span of $\M$.
\end{obs}

\begin{proof}
Follows from Observation~\ref{obssum} because we can turn the entries in the non-diagonal blocks to zeros thanks to the matrix units as in Definition~\ref{defmd}.
\end{proof}

Let us look at the structure of linear combinations of matrices in $\M$.

\begin{lem}\label{lemdiag}
For any $M\in\operatorname{span}\M$ and $i\neq j$, there is a permutation of rows and columns that sends the submatrix $M(a_i\times\Sigma|b_j\times\Sigma)$ to a $\sigma\times\sigma$ block-diagonal matrix with blocks of size $\sigma^{\theta-1}\times\sigma^{\theta-1}$.
\end{lem}

\begin{proof}
Since the functions $\pi^i$ and $\pi^j$ as in Definition~\ref{defmd} are different, we have $\pi^i_\tau=3-\pi^j_\tau$ for some $\tau\in\{1,\ldots,\theta\}$. Therefore, the $\tau$th element in the tuple of $I$-coordinates of any non-zero entry of $M(a_i\times\Sigma|b_j\times\Sigma)$ should coincide with the $\tau$th element in the tuple of its $J$-coordinates. This $\tau$th element can take $\sigma$ different values, which determine a desired partition of the indexing set.
\end{proof}

\begin{obs}\label{obstriv}
Let $S_1,\ldots,S_r$ be disjoint sets of cardinality $c$, and let their union be partitioned into $c$ sets $D_1,\ldots,D_c$ each of which has exactly one element in every of the $S_i$'s. Let us assume that a subset of less than $c/r$ elements was removed from every $S_i$. Then there is a $\delta$ such that none of the elements of $D_\delta$ were removed.
\end{obs}

\begin{proof}
Trivial.
\end{proof}

\begin{lem}\label{lemlemtbm}
Let $\M_0$ be a set of $|K|$ elements in $\operatorname{span}\M$. Then either

\noindent  (1) $\M_0$ contains a matrix of rank at least $\rho$;

\noindent (2) there are $\varphi_1,\ldots,\varphi_r,\psi_1,\ldots,\psi_r\in\Sigma$ such that, for every $M\in\M_0$, the matrix $M(a_1\times\varphi_1,\ldots,a_r\times\varphi_r|b_1\times\psi_1,\ldots,b_r\times\psi_r)$ is a multiple of $W$.
\end{lem}

\begin{proof}
Assuming that (1) is false, we apply Lemma~\ref{lemdiag} to a non-diagonal block $(a_i\times\Sigma|b_j\times\Sigma)$ of a matrix in $\M_0$, and we conclude that at most $\rho\sigma^{\theta-1}$ rows are non-zero in this block. For every fixed $i$, we get an upper bound of $r\rho\sigma^{\theta-1}|K|<\rho^2\sigma^{\theta-1}<|\Sigma|/(2r)$ for the quantity of those elements $\xi\in a_i\times\Sigma$ such that the $\xi$th row is non-zero in at least one of the non-diagonal $(a_i,b_j)$ blocks in at least one of the matrices in $\M_0$. In other words, all the matrices obtained from those in $\M_0$ by removing at most $|\Sigma|/(2r)$ elements from every $a_i\times\Sigma$ will have zero non-diagonal blocks. It remains to find $\varphi_1,\ldots,\varphi_r,\psi_1,\ldots,\psi_r\in\Sigma$ such that the entries $(a_i\times\varphi_i,b_i\times\psi_i)$ are non-zero in the same $M_\Phi$ and do not appear in the removed rows; this is possible by Observation~\ref{obstriv}.
\end{proof}

\begin{cor}\label{corcl1} 
$\min\rank_{\overline{\F}}\,T_\Sigma\,\mathrm{mod}(0,0,\operatorname{span}\M)=\min\rank_{\overline{\F}}\,T\,\mathrm{mod}(0,0,W)$.
\end{cor}

\begin{proof}
The `$\leqslant$' inequality follows from Observation~\ref{obsspan}. In order to prove the `$\geqslant$' part, we note that adding a matrix $M\in\operatorname{span}\M$ to a $3$-slice of $T_\Sigma$ will only increase the rank of the tensor provided that the rank of $M$ is at least $\rho$. Therefore, we can use Lemma~\ref{lemlemtbm} and find, for any tensor $T_0$ of the minimal rank in $T_\Sigma\,\mathrm{mod}(0,0,\operatorname{span}\M)$ and any $i\in I$, $j\in J$, elements $\alpha(i)\in i\times\Sigma$ and $\beta(j)\in j\times\Sigma$ such that the restriction of $T_0$ to $(\alpha(I)|\beta(J)|K)$ belongs to $T\,\mathrm{mod}(0,0,W)$.
\end{proof}

Let us now see how the construction of $\M$ allows us to prove Claim~\ref{claimmain1}. In particular, we can define the tensor $\mathcal{T}$ in the case when $W_1=W_2=0$ and $\dim W_3=1$. In other words, this is the case when $W_3$ is spanned by a single matrix $W'$ which can be written as $PWQ$, where $W$ is the matrix as in the beginning of the section, and $P,Q$ are invertible $I\times I$ and $J\times J$ matrices with entries in $\F$. Namely, we set $\mathcal{T}(T,0,0,W')$ to be the tensor obtained from the $\Sigma$-clone of $T$ by adjoining the matrices $P^\Sigma MQ^\Sigma$ as $3$-slices, where $M$ runs over $\M$ and $P^\Sigma$ stands for the Kronecker product of $P$ and the $\Sigma\times\Sigma$ unity matrix.\footnote{This definition corresponds to writing $T$ with respect to the basis in which $W'$ has the form $W$, then adjoining the matrices in $\M$, and then going back to the initial basis.}

\begin{lem}\label{lemclaim1} 
$\rank_{\overline{\F}}\,\mathcal{T}(T,0,0,W)=\min\rank_{\overline{\F}}\,T\,\mathrm{mod}(0,0,W)+\dim\operatorname{span}\M$.
\end{lem}

\begin{proof}
We replace the first summand in the right-hand side by the left-hand side of the conclusion of Lemma~\ref{corcl1} and apply Lemma~\ref{lemcompl}.
\end{proof}

Now we assume that $(W_1,W_2,W_3)$ are as in Claim~\ref{claimmain1} and $W\in\F^{I\times J}\setminus W_3$; we denote the $\overline{\F}$-linear span of $W_3\cup\{W\}$ by $\mathcal{W}_3$. We set $\mathcal{T}(T,0,0,0)=T$ and recursively define $\mathcal{T}(T,W_1,W_2,\mathcal{W}_3)$ as $\mathcal{T}(\mathcal{T}_0,0,0,\mathcal{W}')$, where $\mathcal{T}_0=\mathcal{T}(T,W_1,W_2,W_3)$ and $\mathcal{W}'$ is the matrix equivalent to the relevant clone of $W$. Observation~\ref{obsspan} and Lemma~\ref{lemclaim1} confirm that $\mathcal{T}(T,W_1,W_2,\mathcal{W})$ is a desired construction for Claim~\ref{claimmain1}.

\section{The proof of Claim~\ref{claimmain2}}

Throughout this section, we assume that $\F$ is an infinite field; we also fix a purely transcendental extension of $\F$ with an infinite basis. The elements of this basis are to be called \textit{variables}, and one may actually think of them as variables changing in $\F$. We denote by $n$ a sufficiently large integer, and we write $r(n)=\lfloor 0.34 n^2\rfloor$. We say that a tensor $T$ with entries in an extension of a field $F$ has $\tau$ \textit{degrees of freedom over} $F$ if the transcendence degree of the extension obtained from $F$ by adjoining the entries of $T$ equals $\tau$. If $\tau$ equals the total number of entries of $T$, then $T$ is called \textit{generic over} $F$. We define the $2n\times 2n\times 2n$ tensor as

$$T=\sum_{\alpha=1}^{r(n)}(x^\alpha,\xi^\alpha)\otimes(y^\alpha,\gamma^\alpha)\otimes(z^\alpha,\zeta^\alpha),$$ where
$x^\alpha,\xi^\alpha,y^\alpha,\gamma^\alpha,z^\alpha,\zeta^\alpha$ are $n$-vectors whose entries are pairwise different variables denoted by $x_i^\alpha,\xi_i^\alpha,y_i^\alpha,\gamma_i^\alpha,z_i^\alpha,\zeta_i^\alpha$---and we denote the algebraic closure of the field obtained from $\F$ by adjoining all these variables by $\Fb$. We define the $n$-element indexing sets $I_1,J_1,K_1$ corresponding to $x_1^\alpha,\ldots,x_n^\alpha$, to $y_1^\alpha,\ldots,y_n^\alpha$, and to $z_1^\alpha,\ldots,z_n^\alpha$, respectively. Similarly, the sets $I_2,J_2,K_2$ correspond to $\xi_1^\alpha,\ldots,\xi_n^\alpha$, to $\gamma_1^\alpha,\ldots,\gamma_n^\alpha$, and to $\zeta_1^\alpha,\ldots,\zeta_n^\alpha$. 
The rank of $T$ does not exceed $r(n)$ for any assignment of values in $\F$ to the variables $x_i^\alpha,\xi_i^\alpha,y_i^\alpha,\gamma_i^\alpha,z_i^\alpha,\zeta_i^\alpha$, so we will conclude the proof of Claim~\ref{claimmain2} if we produce an assignment satisfying
\begin{equation}\label{eqcl2}r(n)<\min\rank_\Fb\,T_{111}\,\mathrm{mod}(U_1,U_2,U_3)+\min\rank_\Fb\,T_{222}\,\mathrm{mod}(V_1,V_2,V_3),\end{equation} where $T_{ijk}$ denotes $T(I_i|J_j|K_k)$, and $U_1$, $U_2$, $U_3$, $V_1$, $V_2$, $V_3$ are the ${\Fb}$-linear spaces spanned, respectively, by the $1$-slices of $T_{211}$, by the $2$-slices of $T_{121}$, by the $3$-slices of $T_{112}$, by the $1$-slices of $T_{122}$, by the $2$-slices of $T_{212}$, and by the $3$-slices of $T_{221}$.

We keep working with generic tensors 
and consider $\Phi\in T_{111}\,\mathrm{mod}\,(U_1,U_2,U_3)$. Every such $\Phi$ can be obtained by adding $\Fb$-linear combinations of the $3$-slices of $T_{112}$ to every of the $3$-slices of $T_{111}$, then $\Fb$-linear combinations of the $2$-slices of $T_{121}$ to the $2$-slices of the resulting tensor, and finally $\Fb$-linear combinations of the $1$-slices of $T_{211}$ to the $1$-slices of what was obtained. Denoting by $A_i^\tau$, $B_j^\tau$, $C_k^\tau$ the coefficients of these linear combinations, we get the expression
$$\varphi_{ijk}=\sum_{\alpha=1}^{r(n)}\left(x_i^\alpha y_j^\alpha z_k^\alpha+a_i^\alpha y_j^\alpha z_k^\alpha+x_i^\alpha b_j^\alpha z_k^\alpha+x_i^\alpha y_j^\alpha c_k^\alpha\right)$$
for the $(i,j,k)$ entry of $\Phi$. Here, $x_i^\alpha$, $y_j^\alpha$, $z_k^\alpha$ are variables defined above, and
$$a_i^\alpha=\sum_{\tau=1}^n A_i^\tau\xi_\tau^\alpha,\,\,\,\,b_j^\alpha=\sum_{\tau=1}^n B_j^\tau\gamma_\tau^\alpha,\,\,\,\,c_k^\alpha=\sum_{\tau=1}^n C_k^\tau\zeta_\tau^\alpha$$
are elements of $\Fb$. We write 
\begin{equation}\label{eqdefk}\mathbb{K}=\F(A_i^\tau,B_i^\tau,C_i^\tau,\xi_i^\alpha,\gamma_i^\alpha,\zeta_i^\alpha),\end{equation} and we are going to prove the following.

\begin{prop}\label{prcl22}
The tensor $\Phi$ has at least $n^3-6\sqrt{3}n^{2.5}$ degrees of freedom over~$\mathbb{K}$.
\end{prop}

Let us see that Proposition~\ref{prcl22} implies Claim~\ref{claimmain2} before we proceed with its proof. We say that a property $\Pi$ holds for \textit{almost all} $m$-vectors over an infinite field $F$ if there is a non-zero polynomial $\psi$ with $m$ variables and coefficients in an extension of $F$ such that $\psi(v)=0$ is a necessary condition for $v$ not to possess $\Pi$.

\begin{lem}\label{fol}
Proposition~\ref{prcl22} implies Claim~\ref{claimmain2}.
\end{lem}

\begin{proof}
We write~\eqref{eqcl2} as a first-order formula $\Psi$ in the language of fields with free variables $\mathcal{V}=(x_i^\alpha,\xi_i^\alpha,y_i^\alpha,\gamma_i^\alpha,z_i^\alpha,\zeta_i^\alpha)$. The quantifier elimination (see Theorem~1.23 in~\cite{QuanElim}) allows us to assume that $\Psi$ is a conjunction of clauses $(f_1*_10)\vee\ldots\vee(f_k*_k0)$, where $*_k$ is either $=$ or $\neq$, and $f_i$ is a polynomial that depends on $\mathcal{V}$ and has integral coefficients. Therefore, the inequality~\eqref{eqcl2} does either hold for almost all assignments of $\mathcal{V}$ or fail for almost all such assignments. Now it is enough to prove~\eqref{eqcl2} for generic $\mathcal{V}$, that is, for the tensor $T$ defined in the beginning of this section. Note that the summands in the right-hand side of~\eqref{eqcl2} are equal, so it remains to prove that $\rank_\Fb \Phi>r(n)/2$. Since the entries of any rank-one $n\times n\times n$ tensor are products of $3n$ elements, the quantity as in Proposition~\ref{prcl22} cannot be less than $3n\rank_\Fb \Phi$. Therefore, Proposition~\ref{prcl22} implies that $\rank_\Fb \Phi$ is at least $n^2/3-2\sqrt{3}n^{2.5}$, which is greater than $r(n)/2$ for large $n$.
\end{proof}

The following is a very well known result in algebraic complexity theory; see Theorem~2.2 in~\cite{Jacobian} for a simple elementary proof in characteristic zero and also Theorem~22, Remark~23 in~\cite{Jacobian2} for a more advanced treatment for positive characteristic. In particular, the first two paragraphs of the proof of Theorem~2.2 in~\cite{Jacobian} remain valid even if one replaces $\C$ in the formulation by arbitrary field $F$; this is so because it is sufficient to consider the case when $F$ is algebraically closed, and in this case the polynomial $G$ as in their proof should be separable.

\begin{lem}\label{lemJacobian}\textup{(See Theorem~2.2 in~\cite{Jacobian}, Remark~23 in~\cite{Jacobian2}.)}
Let $F$ be a field, and let $p_1,\ldots,p_q$ be elements of the algebraic closure of $F(s_1,\ldots,s_q)$. Looking at $p_1,\ldots,p_q$ as algebraic functions of variables $s_1,\ldots,s_q$, we define their Jacobian to be the $q\times q$ matrix with $(i,j)$ entry given by $\left(\partial p_i/\partial s_j\right)$. If this Jacobian has non-zero determinant, then $p_1,\ldots,p_q$ are algebraically independent over $F$.
\end{lem}

\begin{lem}\label{lemgenmatr1}
Let $F$ be an infinite field, let $q<n$ be positive integers, and let $\mu$ be an $n$-vector that has at least $q$ degrees of freedom over $F$. Then the $q$-vector $Q\mu$ is generic over $F$ for almost all $q\times n$ matrices $Q$ with entries in $F$.
\end{lem}

\begin{proof}
We pick a transcendence basis $\mu_{s_1},\ldots,\mu_{s_{k}}$ of the entries of $\mu$. Every coordinate of $\mu$ is now an algebraic function of this basis, so we can define the Jacobian as the $q\times k$ matrix with $(i,j)$ entry given by $\left(\partial (Q\mu)_i/\partial(\mu_{s_j})\right)$. According to Lemma~\ref{lemJacobian}, the Jacobian having rank $q$ is a sufficient condition for $Q\mu$ to be generic. This condition is valid when $Q$ is the matrix with ones at the positions $(t,s_t)$ and zeros everywhere else, so one of the $q\times q$ minors of the Jacobian is really a non-zero polynomial in the entries of $Q$.
\end{proof}

\begin{lem}\label{lemgenmatr}
Let $F$ be an infinite field, let $d<n$ be positive integers, and let $M$ be an $n\times m$ matrix that has at least $mn-\delta$ degrees of freedom over $F$. Let $Q$ be an almost arbitrary $(n-d)\times n$ matrix over $F$. Then there are at most $\delta/d$ columns in $QM$ whose removal leaves a matrix that is generic with respect to $F$.
\end{lem}

\begin{proof}
First, we assume $m=1$; if $d\leqslant \delta$, then there is nothing to prove because the assumption of the lemma allows us to remove the only column of $M$.
Otherwise, we apply Lemma~\ref{lemgenmatr1} and conclude the consideration of the case $m=1$.

Now we assume $m>1$ and proceed by induction; we define $F_j$ and $F_j'$ as the fields obtained from $F$ by adjoining the entries of the first $j$ columns of $M$ and $QM$, respectively. A $\hat{\jmath}$th column is called \textit{weak}  if the transcendence degree of $F_{\hat{\jmath}}$ over $F_{\hat{\jmath}-1}$ is at most $n-d$. The removal of any weak column of $M$ leaves us a matrix having at least $mn-\delta-n+d$ degrees of freedom, and then we can complete the proof by induction. If no weak columns exist, then Lemma~\ref{lemgenmatr1} shows that the $j$th column of $QM$ is generic with respect to $F_{j-1}$ (and consequently with respect to the field $F_{j-1}'\subseteq F_{j-1}$). In particular, the extension $F_j'\supset F_{j-1}'$ has transcendence degree $n-d$, so $F_m'$ has transcendence degree $m(n-d)$ over $F=F_0'$.
\end{proof}

\begin{thr}\label{thrLickt}\textup{(See Theorem~4.4 in~\cite{Lickteig}.)}
Let $u^1,\ldots,u^k,v^1,\ldots,v^k,w^1,\ldots,w^k$ be $n$-vectors whose coefficients are algebraically independent over a field $F$ in common. If $k>n^3/(3n-2)$, then the tensor $\sum_{\alpha=1}^k u^\alpha\otimes v^\alpha\otimes w^\alpha$ is generic over $F$.
\end{thr}

\begin{obs}\label{obs123}
Let $p_1,\ldots,p_q$ be polynomials of the same total degree $d$ over a field $F$. If $(p_1^{in},\ldots,p_q^{in})$ are generic over $F$, then $(p_1,\ldots,p_q)$ are generic over $F$ as well, where $p^{in}$ denotes the sum of all monomials of $p$ whose degree is maximal.
\end{obs}

\begin{proof}
If $h(p_1,\ldots,p_q)=0$, then $h^{in}(p_1^{in},\ldots,p_q^{in})=0$.
\end{proof}


Let $P=\sum_\alpha u^\alpha\otimes v^\alpha\otimes w^\alpha$ be an $n\times n\times n$ tensor over a field $F$, and let $Q$ be a $q\times n$ matrix over a subfield $K\subset F$. We define $P_Q=\sum_\alpha (Qu^\alpha)\otimes (Qv^\alpha)\otimes (Qw^\alpha)$ and we note that $P_Q$ is obtained from $T$ by a sequence of $K$-linear transformations of its $1$-slices, $2$-slices, and $3$-slices. In particular, $P$ has at least as many degrees of freedom over $K$ as $P_Q$ does have.
Now we are ready to prove Proposition~\ref{prcl22}.


\begin{proof}[The proof of Proposition~\ref{prcl22}.]
Let $M$ be the $n\times 3r(n)$ matrix formed with the vectors $x^\alpha,y^\alpha,z^\alpha$ as in the definition of $T$. Since the field $\mathbb{K}$ defined in~\eqref{eqdefk} has at most $3nr(n)+3n^2$ degrees of freedom over $\F$, and since $\mathbb{K}(x^\alpha,y^\alpha,z^\alpha)$ contains $\F(x^\alpha,y^\alpha,z^\alpha,\xi^\alpha,\gamma^\alpha,\zeta^\alpha)$, the matrix $M$ has at least $6nr(n)-(3nr(n)+3n^2)=3nr(n)-3n^2$ degrees of freedom over $\mathbb{K}$. Now we take $d=\lfloor \sqrt{3n}\rfloor$ and apply Lemma~\ref{lemgenmatr} to the matrix $M$ and field $\mathbb{K}$; we denote the vector $Q\pi^\alpha$ by $\overline{\pi}^\alpha$, where $\pi$ may stand for $a,b,c,x,y,z$. The $(n-d)\times(n-d)\times(n-d)$ tensor $\Phi_Q$ equals 
\begin{equation}\label{eqprcl22}\sum_{\alpha=1}^{r(n)}\left(\overline{x}^\alpha \overline{y}^\alpha \overline{z}^\alpha+\overline{a}^\alpha \overline{y}^\alpha \overline{z}^\alpha+\overline{x}^\alpha \overline{b}^\alpha \overline{z}^\alpha+\overline{x}^\alpha \overline{y}^\alpha \overline{c}^\alpha\right).\end{equation}

We denote by $\mathcal{A}$ the set of all $\alpha$ such that none of the three columns $x^\alpha, y^\alpha, z^\alpha$ were identified as `weak' and removed while applying Lemma~\ref{lemgenmatr}. According to this lemma, the union of the vectors $\overline{x}^\alpha, \overline{y}^\alpha, \overline{z}^\alpha$ with $\alpha\in\mathcal{A}$ is generic over $\mathbb{K}$. The number of weak columns is at most $3n^2/d$, so that $|\mathcal{A}|>n^2/(3n-2)$, and the $\mathcal{A}$-part of the sum in~\eqref{eqprcl22} is a tensor generic over $\mathbb{K}$ according to Theorem~\ref{thrLickt} and Observation~\ref{obs123}. (In order to use these, we look at the coordinates of $\overline{x}^\alpha, \overline{y}^\alpha, \overline{z}^\alpha$ with $\alpha\in\mathcal{A}$ as variables and note that $\overline{a}^\alpha, \overline{b}^\alpha, \overline{c}^\alpha$ have coordinates in $\mathbb{K}$.) 
Therefore, $\Phi_Q$ is the sum of a tensor generic over $\mathbb{K}$ and a tensor of rank at most $3n^2/d$, so it has at least $(n-d)^3-3n(3n^2/d)\geqslant n^3-6\sqrt{3}n^{2.5}$ degrees of freedom over $\mathbb{K}$.
\end{proof}

Therefore, we completed the proof of Proposition~\ref{prcl22}, which implies Claim~\ref{claimmain2} according to Lemma~\ref{fol}. Since Claim~\ref{claimmain1} was proved in the previous section, we can use Theorem~\ref{thrmain} and complete the proof of Theorem~\ref{thrmain2}.

\section{Concluding remarks}

As said above, the author disproved the direct sum conjecture in the original formulation, that is, over any infinite field. In fact, his proof is going to work for sufficiently large finite fields as well (although he did not try to give an effective upper bound on the size of the smallest field for which it may not work). The author thinks that the direct sum conjecture is false over any field, but in order to construct a counterexample in the way similar to this paper, one would need to estimate the cardinalities of the arising sets of matrices instead of the corresponding dimensions as it is done in the current proof. Actually, it seems possible to achieve a counterexample in this way despite a much greater amount of technical difficulties.

A multidimensional analogue of the direct sum conjecture states that the rank of $d$-dimensional tensors is additive with respect to direct sums. This problem remains open for $d\geqslant 4$, and I am not sure that the present approach can lead to a progress on the multidimensional version. The analogous statement but restricted to symmetric tensors is open already for $d\geqslant 3$, see~\cite{CarCatChi, LandBook, Teit} for a review of the current state of art and new results on this problem. We also mention the concept of \textit{border rank} of a complex tensor, which is the smallest $r$ such that a given tensor is the limit of a sequence of tensors of rank $r$; the direct sum conjecture is known to fail for border rank of $d$-tensors at least when $d=3$, see the original construction in Section~6 of~\cite{Sch} and a discussion in Chapter~11 of~\cite{LandBook}.

The author would like to thank Mateusz Micha\l{}ek for pointing him out to some evidence against the direct sum conjecture, which includes a counterexample for the border rank analogue of it. This has happened when the author was visiting the Max Planck Institute for Mathematics in the Sciences in Leipzig, and he would like to thank Mateusz for invitation and the colleagues from the institute for hospitality. The idea of Claim~\ref{claimmain1} came to me when I was working on a revision of the paper~\cite{myComon}, which contains a related construction. Therefore, I would also like to thank the reviewers of \textit{SIAM Journal of Applied Algebraic Geometry} for helpful comments on~\cite{myComon}, which may have stimulated the development of the approach of this paper.

\end{document}